\documentclass{amsart}

\usepackage{textcomp}

\usepackage{algorithmic, algorithm}
\usepackage{caption}

\usepackage{amssymb,latexsym,amsfonts,amsmath}
\usepackage{graphicx}
\include{diagrams}
\usepackage{eso-pic}
\usepackage{graphicx}
\usepackage{color}
\usepackage{type1cm}

\makeatother

\usepackage{amssymb,latexsym,amsfonts,amsmath}
\usepackage{graphicx}

\newtheorem{theorem}{Theorem}[section]
\newtheorem{lemma}[theorem]{Lemma}

\newtheorem{problem}[theorem]{Problem}
\newtheorem{proposition}[theorem]{Proposition}
\newtheorem{corollary}[theorem]{Corollary}

\theoremstyle{definition}
\newtheorem{definition}[theorem]{Definition}

\theoremstyle{remark}

\numberwithin{equation}{section}

\newcommand{\R}{{\mathbb{R}}}

\newcommand{\N}{{\mathbb{N}}}

\begin{document}
\begin{abstract}
Cloud computing platforms are being increasingly used for closing feedback control loops, especially when computationally expensive algorithms, such as model-predictive control, are used to optimize performance. Outsourcing of control algorithms entails an exchange of data between the control system and the cloud, and, naturally, raises concerns about the privacy of the control system's data
(e.g., state trajectory, control objective). Moreover, any attempt at enforcing privacy needs to add minimal computational overhead to avoid degrading control performance. In this paper, we propose several transformation-based methods for enforcing data privacy. We also quantify the amount of provided privacy and discuss how much privacy is lost when the adversary has access to side knowledge. We address three different scenarios: a) the cloud has no knowledge about the system being controlled; b) the cloud knows what sensors and actuators the system employs but not the system dynamics; c) the cloud knows the system dynamics, its sensors, and actuators. In all of these three scenarios, the proposed methods allow for the control over the cloud without compromising private information (which information is considered private depends on the considered scenario).
\end{abstract}

\title[Symmetries and isomorphisms for privacy in control over the cloud]{Symmetries and isomorphisms for privacy in control over the cloud*}
\thanks{*The work of the authors was partially supported by the NSF grants 1740047, 1705135 and by the Army Research Laboratory under Cooperative Agreement W911NF-17-2-0196. The views and conclusions contained in this document are those of the authors and should not be interpreted as representing the official policies, either expressed or implied, of the Army Research Laboratory or the U.S. Government. The U.S. Government is authorized to reproduce and distribute reprints for Government purposes notwithstanding any copyright notation here on. }

\author[A. Sultangazin and P. Tabuada]{Alimzhan Sultangazin and Paulo Tabuada}
\address{Department of Electrical Engineering\\
University of California at Los Angeles,
Los Angeles, CA 90095}
\email{\{asultangazin, tabuada\}@ucla.edu}

\maketitle

\section{Introduction}

\subsection{Motivation}
The recent advances in reliability and speed of communication have led to an increased use of cloud-based services, which provide computation and data storage capabilities to clients. Control over the cloud \cite{7575205,6908023,5871633} has numerous advantages, which include easier installation and maintenance \cite{cdc1}, and the availability of global information from all of the cloud's clients when making control decisions. However, the main advantage of control over the cloud is that it allows control systems to outsource expensive computational tasks to the cloud, thus potentially improving the speed of computation and freeing the local computational capabilities for other tasks.

An illustrative example of the benefits of outsourcing computing can be observed in Model Predictive Control (MPC). MPC is a conceptually simple, yet powerful scheme that was adopted in industry for multivariable control \cite{mpc}. MPC inherently involves solving complex constrained optimization problems \emph{on-line} (i.e., within one sampling interval). The work in \cite{7575205} presents an experimental study that shows feasibility of MPC over the cloud for robot control. Another work (see \cite{6908023}) considered the practicality and benefits of cloud-based MPC for a large-scale solar plant. The availability of global information provided by control over the cloud can have many practical benefits, as shown in \cite{5871633}. There, the authors propose a solution to the problem of traffic flow estimation via the cloud.

However, relying on a third-party to perform computation is not without its dangers. Despite the benefits of control over the cloud, a number of studies have shown that exposing existing systems to connectivity may lead to security vulnerabilities in a vast variety of applications \cite{Checkoway:2011:CEA:2028067.2028073, Gollmann:2015:CSS:2732198.2732208, 4596535, Ghena:2014:GLF:2671293.2671300}, including control of process plants, traffic infrastructure, and smart meter systems. Cyber-security attacks vary based on the amount of resources the attacker possesses \cite{teixeira}. One of the most basic attacks that requires little resources is eavesdropping. It can often serve as a stepping stone in the implementation of more complex attacks \cite{sinopoli}. In control over the cloud, eavesdropping involves the adversary listening in to the communication channel between sensors, controllers, and actuators to leak valuable information about the model, the controller, and trajectories \cite{farokhi}. The client is expected to disclose all of this sensitive information to the cloud if it intends to receive valid control inputs from it. For example, we would expect drivers to share their locations, final destinations and, perhaps, dynamics to successfully allow traffic control over the cloud.

Eavesdropping attacks are usually prevented with encryption - the plant and the cloud establish a shared key with which they encrypt transmitted messages and decrypt the received ones. However, if the adversary manages to undermine the security of the cloud (e.g., gain unauthorized access to its memory), this technique can no longer protect the system since the cloud accesses the decrypted data. As stated in \cite{7011179}, traditional IT security provides only a partial solution. Therefore, there is a pressing need for development of control-over-the-cloud methods that do not rely on decryption of the incoming data. Although much effort has been directed to this problem, a universally secure scheme for control over the cloud that could support any client functionality has not yet been created \cite{impossible, andreea}. When solving the problem of private control over the cloud, two other important concerns need to be accounted for: efficiency and safety. Privacy cannot come at the cost of degradation of control performance either due to delays in the feedback loop or inaccurate control inputs.

\subsection{Related work}

The body of work on privacy in control over the cloud can be categorized into methods based on homomorphic encryption, differential privacy, and algebraic transformations.

When using homomorphic encryption techniques, the cloud is able to perform the computations on encrypted data without the need to decrypt it \cite{Armknecht2015AGT}. As a result, the cloud can implement optimization algorithms using fully homomorphic encryption (FHE), as done in \cite{7403296, 7244402}. However, FHE is inefficient in terms of execution time \cite{Armknecht2015AGT}, which makes it impractical for online optimization. Therefore, partially homomorphic encryption (PHE), a simpler form of FHE only allowing for a subset of operations to be performed on encrypted data, has become more popular in connection to privacy in control over the cloud. While PHE methods are shown to be feasible and are able to provide privacy guarantees \cite{FAROKHI2016163,farokhi, cdc1,7799042,andreea, 2018arXiv180902267A}, execution time, which grows disproportionally with an increase in key length \cite{farokhi, andreea}, remains a valid concern in these methods. A consequence of this is that using homomorphic encrypion may potentially lead to instability in the controlled system due to processing delays. To address this problem, some works (see \cite{farokhi}) have shown that encryption parameters can be chosen to ensure stability of the closed-loop performance, thus providing a natural trade-off between security and control performance. The practical feasibility of encrypted control systems has been validated in \cite{8452372} by considering control of a DC motor in real time.

Inspired by studies in privacy of databases, the problem of privacy in control over the cloud has also been approached from the standpoint of differential privacy (see \cite{7798915, 8263806}). This technique ensures that the risk of losing privacy of a single user\textquotesingle s data by means of data queries is low. The main idea of these methods is to perturb the response to a data query with appropriate noise \cite{10.1007/11787006_1}. However, to achieve more privacy, the user must sacrifice accuracy (i.e., add more noise), which, in the context of control, degrades the control performance.

The ideas behind algebraic transformation methods have initially stemed from works on privacy in optimization. The idea is to use algebraic transformations to produce a different, but equivalent optimization problem. In other words, although the cloud does not know the original optimization problem, it can provide the client with an optimal solution to an equivalent optimization problem from which the client is able to recover the optimal solution to the original problem. Although initially these methods found application exclusively in linear programs \cite{mangasarian, 5935305}, several efforts have been directed to providing a unified framework and generalizing them to convex optimization problems (see \cite{weeraddana, WEERADDANA20179502}). The work in \cite{weeraddana} also shows one of the first attempls to define and \emph{quantify} privacy of transformation-based methods. Algebraic transformation methods found applications in control due to their efficiency and guaranteed optimality of the solution \cite{WEERADDANA20179502}. For example, in \cite{Xu:2015:SRC:2808705.2808708} the authors propose a hybrid transformation-based method to preserve privacy of an MPC controller in networked control systems. In \cite{7438892}, transformation-based methods are used to provide privacy in a specific problem AC Optimal Power Flow.

\subsection{Contributions}
This paper focuses on the use of transformation-based methods to preserve privacy of the system dynamics, control objective and constraints, and system trajectories. The contributions of this paper are fourfold:
\begin{enumerate}
\item we propose using isomorphisms and symmetries of control systems as a source of transformations so as to keep data private;
\item we quantify the privacy guaranteed by these methods via the dimension of the set that describes the uncertainty experienced by the adversary;
\item we quantify how much privacy is lost when the adversary is assumed to have access to side knowledge;
\item we show that the proposed method is computationally light as it only requires matrix multiplications.
\end{enumerate}
The method proposed in this paper was initially introduced in \cite{alimzhan}. In \cite{8487355}, it was extended to networked control systems with several agents requesting control input from a single cloud. In \cite{Sultangazin2019}, the dimension of the set describing the uncertainty experienced by the adversary was proposed as a measure of privacy for this method and was evaluated for the special case of free group actions. This paper provides a unified presentation of the results in \cite{alimzhan, Sultangazin2019} with simpler proofs and several new results, such as the bounds on privacy when the group action is not free and an exact quantification of privacy for prime systems.

While privacy quantification in optimization has been studied in \cite{WEERADDANA20179502}, this work considers how much privacy is preserved in the more challenging context of control. Moreover, the measure of privacy proposed in this work has been chosen to be suitable for problems of optimization in control systems and, therefore, is different from any of those proposed in \cite{WEERADDANA20179502}. Although the application of transformation-based methods in control has been previously discussed in \cite{Xu:2015:SRC:2808705.2808708}, the scheme proposed there only considers a special case, where the cloud optimizes the weighted sum of the norms of the input and state, and the state is taken to be the output of the system. Our algorithm can be applied to a wider class of problems as we allow for arbitrary quadratic costs, linear constraints and outputs different from the state.

\section{Problem Definition}
\label{problem}
\subsection{Plant dynamics and control objective}

We consider discrete-time affine plants, denoted by $\Sigma$, and described by:
\begin{equation}
\begin{aligned}
\Sigma : \quad
\end{aligned}
\begin{aligned}
\label{eq:tru_sys}
\bar{x}_{k+1} & =\bar{A} \bar{x}_k+\bar{B} u_k+\bar{c}\\
\bar{y}_k & = \bar{C} \bar{x}_k+ \bar{d}, 
\end{aligned}
\end{equation}
where $\bar{A} \in\mathbb{R}^{n\times n}$, $\bar{B} \in\mathbb{R}^{n\times m}$,
$\bar{C} \in\mathbb{R}^{p\times n}$, $\bar{c} \in\mathbb{R}^{n}$, and $\bar{d} \in\mathbb{R}^{p}$ describe the dynamics of the system,
and $\bar{x}_k \in \mathbb{R}^{n}$, $u_k \in \mathbb{R}^{m}$ and $\bar{y}_k \in \mathbb{R}^{p}$
denote the state, input and output of the system at time $k$, respectively. We assume that system $\Sigma$ is controllable and observable. We also assume, without loss of generality, that $\text{ker }\bar{B} = \{ 0 \}$ and $\text{Im } \bar{C} = \R^p$, since we can always eliminate linearly dependent columns (resp. rows) from $\bar{B}$ (resp. $\bar{C})$.

To simplify notation, we lift every affine map $Wx+v$ to a linear map through the following construction:
\begin{equation}
\begin{aligned}
\label{eq:affine-linear}
Wx + v \mapsto \begin{bmatrix} W & v \\ 0 & 1 \end{bmatrix} \begin{bmatrix} x \\ 1 \end{bmatrix}.
\end{aligned}
\end{equation}

Applying \eqref{eq:affine-linear} to \eqref{eq:tru_sys}:
\begin{equation}
\begin{aligned}
\label{eq:tru_sys2}
x_{k+1} &\triangleq
\begin{bmatrix}
\bar{x}_{k+1}\\1
\end{bmatrix} = 
\begin{bmatrix}
\bar{A} & \bar{c}\\
0 & 1
\end{bmatrix}
\begin{bmatrix}
\bar{x}_k\\1
\end{bmatrix} + 
\begin{bmatrix}
\bar{B}\\0
\end{bmatrix}u_k \\
 &\triangleq Ax_k + Bu_k \\
y_k&\triangleq \begin{bmatrix} \bar{y}_k\\1\end{bmatrix}  = 
\begin{bmatrix}
\bar{C} & \bar{d} \enspace \\ 0 & 1 \enspace
\end{bmatrix}
\begin{bmatrix}
\bar{x}_k\\1
\end{bmatrix}\triangleq Cx_k.
\end{aligned}
\end{equation}

In the remainder of the paper we suppress the inner structure for simplicity and represent all the systems in the linear form \eqref{eq:tru_sys2}. However, the reader is advised to remember that we are dealing with affine maps. This is also true for the affine maps we will use to define isomorphisms. 

We refer to system \eqref{eq:tru_sys2} as the triple $\Sigma=(A,B,C)$. We call a triple $\lbrace x_k,u_k,y_k \rbrace_{k\in\N}$ a trajectory of $\Sigma$ if it satisfies \eqref{eq:tru_sys} for all $k \in \mathbb{N}$.

Additionally, we define a cost function \mbox{$J: \mathbb{R}^n \times (\mathbb{R}^m)^{N+1} \rightarrow \mathbb{R}$} for $N\in \mathbb{N}\cup \lbrace +\infty \rbrace$ that allows to compare trajectories and, thus, to formulate different control objectives. In alignment with the linear framework, we consider quadratic cost functions given by:

\begin{align}
\label{eq:cost}
J(x,u) = & \sum^{N}_{i=0}\Delta \eta_i^T M \Delta \eta_i,
\end{align}
where $\Delta\eta_i=\begin{bmatrix}x_i-x^*_i & u_i-u^*_i\end{bmatrix}^T$, $x = \{ x_0, ..., x_N \}$ and $u = \{ u_0, ..., u_N \}$. The sequences $x^* = \{ x_0^*, ..., x^*_N \}$ and $u^* = \{ u_0^*, ..., u^*_N \}$ denote the reference trajectories to be tracked. We define $M \in \R^{(n+m+1)\times (n+m+1)}$ to be a positive-definite matrix. Due to the lift \eqref{eq:affine-linear}, this cost includes not only quadratic, but also linear terms.

In addition to a cost, we also consider control objectives that require certain constraints to be satisfied at all times. These constraints are defined as:
\begin{align}
\label{eq:constraint}
D\eta_i \leq 0,\quad \forall i \in \{ 0, 1, ..., N\},
\end{align}
where $\eta_i=\begin{bmatrix} x_i & u_i\end{bmatrix}^T$ and $D \in \R^{h \times (n+m+1)}$. Note that, despite appearing to be linear constraints, the constraints above are in fact affine, in view of the construction \eqref{eq:affine-linear}.

\subsection{Attack model and privacy objectives}

The cloud is treated as a curious but honest adversary: the cloud adheres to the computations prescribed by an agreed-upon protocol, but may seek to extract and leak confidential information by keeping record of all computations and communicated messages.

The interaction between the plant and the cloud is performed in two steps. During the first step, called the handshaking, the plant provides the cloud with a suitably modified version of the plant model, cost, and constraints. In exchange, the cloud agrees to compute the input minimizing the provided cost, subject to the constraints and plant dynamics. During the second step, called plant execution, the plant repeatedly sends a suitably modified version of its measurements to the cloud. The cloud computes a new input based on the received measurements and sends it to the plant, where it is suitably modified before being applied to the plant.

In the previous paragraph we purposely used the vague expression ``suitably modified''. Making this expression more concrete requires that we first define the knowledge available to the plant. We consider the following three scenarios.

\begin{problem}[Scenario 1]
\label{scen_1}
Assuming the cloud has no knowledge about the plant:
\begin{enumerate}
\item how to modify the plant $(A,B,C)$, cost $J$, and constraint matrix $D$ before sending them during the handshaking step,
\item how to modify the measurements sent to the plant, and 
\item how to modify the inputs received from the plant,
\end{enumerate}
so that the plant's trajectory minimizes cost $J$ in \eqref{eq:cost}, while preventing the cloud from learning the plant $(A,B,C)$, the cost $J$, the constraint matrix $D$, and the plant's trajectory $\{x_k,u_k,y_k\}_{k\in\N}$?
\end{problem}

\begin{problem}[Scenario 2]
\label{scen_2}
Assuming the cloud has no knowledge about the plant except for knowing what are its sensors and actuators:
\begin{enumerate}
\item how to modify the plant $(A,B,C)$, cost $J$, and constraint matrix $D$ before sending them during the handshaking step; 
\item how to modify the measurements sent to the plant, and 
\item how to modify the inputs received from the plant,
\end{enumerate}
so that the plant's trajectory minimizes cost $J$ in \eqref{eq:cost}, while preventing the cloud from learning the plant $(A,B,C)$, the cost $J$, the constraint matrix $D$, and the plant's trajectory $\{x_k,u_k,y_k\}_{k\in\N}$?
\end{problem}

\begin{problem}[Scenario 3]
\label{scen_3}
Assuming the cloud has complete knowledge about the plant dynamics, including its sensors and actuators:
\begin{enumerate}
\item how to modify cost $J$, and constraint matrix $D$ before sending them alongside the plant $(A,B,C)$ during the handshaking step; 
\item how to modify the measurements sent to the plant, and 
\item how to modify the inputs received from the plant,
\end{enumerate}
so that the plant's trajectory minimizes cost $J$ in \eqref{eq:cost}, while preventing the cloud from learning the cost $J$, the constraint matrix $D$, and the plant's trajectory $\{x_k,u_k,y_k\}_{k\in\N}$?
\end{problem}

These problems are solved in Section \ref{privacy_sol} by utilizing isomorphisms and symmetries of control systems we define next in Section \ref{iso_sym}.

\section{Isomorphisms and\\symmetries of control systems} 
\label{iso_sym}

In this section, we introduce the notions of isomorphism and symmetry of control systems along with several technical results used in Section \ref{privacy_sol} to provide a solution to the problems described in Section \ref{problem}.

Let us denote by $\mathcal{S}_{n,m,p}$ the set of all controllable and observable linear control systems with state, input and output dimensions $n$, $m$, and $p$, respectively.

\begin{definition}
\label{def:isomorphism}
An isomorphism of control systems in $\mathcal{S}_{n,m,p}$ is a quadruple $\psi=(P,F,G,S)$ consisting of a change of state coordinates $P:\R^n\to \R^n$, state feedback $F:\R^n\to \R^m$, a change of coordinates in the input space $G:\R^m\to \R^m$, and a change of coordinates in the output space $S:\R^p\to\R^p$. Transformations $P$ and $S$ are affine invertible maps, $F$ is an affine map and $G$ is a linear invertible map.
\end{definition}

Recall that, to simplify notation, we lift the affine maps to linear maps using the transformation \eqref{eq:affine-linear}.

Let us also denote the set of isomorphisms of $\mathcal{S}_{n,m,p}$ described in Definition \ref{def:isomorphism} as $\mathcal{G}_{n,m,p}$. The set $\mathcal{G}_{n,m,p}$ forms a group under function composition as the group operation\footnote{A composition of two isomorphisms is given by $\psi_{2}\circ \psi_{1}=(P_2P_1, G_2F_1+F_2P_1,G_2G_1,S_2S_1)$, the identity is $\psi_e=(I,0,I,I)$ and the inverse is given by $\psi^{-1}=(P^{-1},-G^{-1}FP,G^{-1},S^{-1}).$}. This allows us to define a group action of $\mathcal{G}_{n,m,p}$ on the set of linear control systems $\mathcal{S}_{n,m,p}$.

\begin{definition}
Each element $\psi \in \mathcal{G}_{n,m,p}$ acts on $\Sigma \in \mathcal{S}_{n,m,p}$ to produce $\psi_* \Sigma$ given by:
\begin{equation}
\begin{aligned}
\label{eq:psi}
\psi_*\Sigma&=(P,F,G,S)_*(A,B,C) \\&=(P(A-BG^{-1}F)P^{-1},PBG^{-1},SCP^{-1})\\& \triangleq (\tilde{A}, \tilde{B}, \tilde{C}) \triangleq \tilde{\Sigma}.
\end{aligned}
\end{equation}
The map $\psi_*$ is called an isomorphism action. We also say that systems $\Sigma$ and $\tilde{\Sigma}$ are equivalent.
\end{definition}

An isomorphism maps the state $x_k$, input $u_k$, and output $y_k$ of system $\Sigma$ to the state $\tilde{x}_k$, input $\tilde{u}_k$, and output $\tilde{y}_k$ of system $\tilde{\Sigma}$ as follows:
\begin{align}
\label{eq:tran_state} \tilde{x}_k & =Px_k\\
\label{eq:tran_input} \tilde{u}_k & =Fx_k+Gu_k\\
\label{eq:tran_output} \tilde{y}_k & =Sy_k.
\end{align}
Similarly, an isomorphism induces transformation on the control objectives --- i.e., the cost and constraints. The effect of $\psi$ on $\eta_k$ can be represented by:
\begin{align}
\tilde{\eta}_k=\begin{bmatrix} \tilde{x}_k\\  \tilde{u}_k\end{bmatrix}=\begin{bmatrix}P & 0 \\ F & G\end{bmatrix}\begin{bmatrix} x_k\\  u_k\end{bmatrix}\triangleq L \eta_k.
\end{align}
Therefore, the cost function $J$ can be expressed as a function of the sequence of modified states $\tilde{x} = \{ \tilde{x}_0, ..., \tilde{x}_N \}$ and the sequence of modified inputs $\tilde{u} = \{ u_0, ..., \tilde{u}_N \}$ as follows:
\begin{align}
\label{eq:fake_cost}
\tilde{J}(\tilde{x},\tilde{u})=& \psi_*J\left(x,u\right) =\sum_{i=0}^{N} \Delta \tilde{\eta}_i^T \tilde{M} \Delta \tilde{\eta}_i,
\end{align}
where $\tilde{M}=L^{-T} M L^{-1}$.
Applying the isomorphism action to the constraints in \eqref{eq:constraint} yields:
\begin{align}
\label{eq:fake_const}
\tilde{D} \tilde{\eta}_i \leq 0, \quad \forall i \in \{ 0, 1, ..., N\},
\end{align}
where $\tilde{D}=\psi_*D=DL^{-1}$.

The effect of an isomorphism on the system, trajectory, cost and constraints will be used in Section \ref{privacy_sol} to prevent the cloud from learning them.

For a given system $\Sigma$, there is a special subgroup of $\mathcal{G}_{n,m,p}$ called the symmetry group of $\Sigma$, which is defined by the following property.
\begin{definition}
Let $\Sigma \in \mathcal{S}_{n,m,p}$. An isomorphism $\psi \in \mathcal{G}_{n,m,p}$ is said to be a symmetry of $\Sigma$ if $\psi_* \Sigma = \Sigma$. The subgroup of symmetries of $\Sigma$ is denoted here as $\mathcal{K}_{n,m,p}(\Sigma)$.
\end{definition}

The notion of isomorphism was crafted to preserve properties of control systems. Among these, trajectories have a special significance.
A simple induction argument can be used to establish the following result.

\begin{lemma}
\label{th:trajectories}
Let $\Sigma \in \mathcal{S}_{n,m,p}$ and $\psi \in \mathcal{G}_{n,m,p}$. If $\tilde{\Sigma}=\psi_*{\Sigma}$ and $\{x_k,u_k,y_k\}_{k\in\N}$ is a trajectory of $\Sigma$, then $\{\tilde{x}_k, \tilde{u}_k, \tilde{y}_k\}_{k\in\N}$, as given by \eqref{eq:tran_state} - \eqref{eq:tran_output}, is a valid trajectory of $\tilde{\Sigma}$. 
\end{lemma}

This means that if the cloud receives $\tilde{\Sigma}$ during the handshaking step, then the received sequence of measurements $\tilde{y}$ and the produced sequence of control inputs $\tilde{u}$ in the subsequent execution step are compatible with the plant $\tilde{\Sigma}$. To elaborate, both the modified measurements $\tilde{y}$ and modified control inputs $\tilde{u}$ would be compatible with modified dynamics $\tilde{\Sigma}$.

Let us now define $\bar{\mathcal{S}}_{n,m,p}$ to be a set of quadruples $\Omega \triangleq \left\{ \Sigma, J, D, \left\{x_k, y_k, u_k \right\} \right\}_{k \in \N}$ such that $\left\{x_k, y_k, u_k \right\}$ is a trajectory of a linear system $\Sigma \in \mathcal{S}_{n,m,p}$ minimizing cost function $J$ under constraints $D$. 

\begin{lemma}
\label{th:smooth_manifold}
The set $\bar{\mathcal{S}}_{n,m,p}$ is a smooth manifold.
\end{lemma}
\begin{proof}
We can see that $\bar{\mathcal{S}}_{n,m,p}$ is, in fact, the Cartesian product of $\mathcal{S}_{n,m,p}$ with the set of cost functions $\mathcal{M}^{++}(m+n+1, \R)$, defined by positive-definite matrices, with the set of constraints $\mathcal{M}_{d}(h \times (m+n+1), \R)$, defined by the set of full-rank matrices, where $d=\text{min}\{h, m+n+1\}$. It is known that the product space is a smooth manifold if its constituents are smooth manifolds \cite[p.~21]{manifolds}. It remains to show that these constinuents are indeed smooth manifolds.

Let us construct the map:
\begin{equation}
\begin{aligned}
f_S: &\R^{n \times (n+1)} \times \R^{n \times m} \times \R^{p \times (n+1)} \rightarrow \R^2 \\ &(A,B,C) \mapsto \left( \text{det }\mathcal{C}, \text{det } \mathcal{O} \right) ,
\end{aligned}
\end{equation}
where $\mathcal{C}$ and $\mathcal{O}$ are the controllability and observability matrices of the dynamics $(A,B,C)$. It can be seen that $\mathcal{S}_{n,m,p} = f_S^{-1}(\R^2 \setminus \left( 0, 0 \right) )$. The function $f_S$ is continuous since each of its elements is defined by a polynomial function of the elements of $(A,B,C)$. Given that for continuous functions the preimage of every open set is an open set, we have that $\mathcal{S}_{n,m,p}$ is an open subset of the domain of $f_S$. Seeing that the domain of $f_S$ is a smooth manifold, $\mathcal{S}_{n,m,p}$ is a smooth manifold of dimension $n(n+1) + nm + p(n+1)$.

The set of positive-definite matrices $\mathcal{M}^{++}(m+n+1,\R)$ is shown to be a smooth embedded submanifold of $\R^{(m+n+1) \times (m+n+1)}$ of dimension $(m+n+1)(m+n+2)/2$ in \cite{5278558}.

The set of full-rank matrices $\mathcal{M}_{d}(h \times (m+n+1), \R)$ is a smooth manifold of dimension $h(m+n+1)$ \cite[p. 19]{manifolds}.
\end{proof}

Similarly to $\mathcal{S}_{n,m,p}$, we can define a group action of $\mathcal{G}_{n,m,p}$ on $\bar{\mathcal{S}}_{n,m,p}$ in view of the previous discussion.

Therefore, we can use the isomorphism action of $\mathcal{G}_{n,m,p}$ to define an equivalence relation on $\bar{\mathcal{S}}_{n,m,p}$.
\begin{definition}
Let $\Omega=(\Sigma,J,D,\{x_k,u_k,y_k\}_{k\in\N})$ and $\tilde{\Omega}=(\tilde{\Sigma},\tilde{J},\tilde{D},\{\tilde{x}_k,\tilde{u}_k,\tilde{y}_k\}_{k\in\N})$ be elements of $\bar{\mathcal{S}}_{n,m,p}$.
The equivalence relation $\sim_{\mathcal{G}}$ on $\bar{\mathcal{S}}_{n,m,p}$ denoted by:
\begin{align}
&\Omega \sim_{\mathcal{G}} \tilde{\Omega},
\end{align}
is defined by the existence of $\psi \in \mathcal{G}_{n,m,p}$ such that:
\begin{align}
& \tilde{\Omega}=\psi_* \Omega;
\end{align}
i.e., $\tilde{\Sigma}=\psi_*\Sigma$, $\tilde{J}=\psi_*J$, $\tilde{D}=\psi_*D$, and \mbox{$\{\tilde{x}_k,\tilde{u}_k,\tilde{y}_k\}_{k\in\N}$} is given in terms of $\{x_k,u_k,y_k\}_{k\in\N}$ as in \eqref{eq:tran_state} - \eqref{eq:tran_output}.
\end{definition}

The equivalence relation $\sim_\mathcal{G}$, in turn, defines equivalence classes in $\bar{\mathcal{S}}_{n,m,p}$. The equivalence class of $\Omega \in \bar{\mathcal{S}}_{n,m,p}$ defined by the action of $\mathcal{G}_{n,m,p}$ is the set:
\begin{align}
[ \Omega ] &\triangleq \{ \Omega' \in  \bar{\mathcal{S}}_{n,m,p}| \exists \psi \in \mathcal{G}_{n,m,p} \text{ such that } \Omega'=\psi_* \Omega \} \nonumber \\
&= \{ \psi_* \Omega | \psi \in \mathcal{G}_{n,m,p} \}.
\end{align}
This equivalence class is also called the orbit of $\Omega$ under action of $\mathcal{G}_{n,m,p}$.

To facilitate further results, let us show that $\mathcal{G}_{n,m,p}$ is a Lie group acting on $\bar{\mathcal{S}}_{n,m,p}$.

\begin{lemma}
\label{th:lie_group}
The group $\mathcal{G}_{n,m,p}$ is a Lie group of dimension $n(n+1)+m(n+1)+m^2 + p(p+1)$ acting smoothly on $\bar{\mathcal{S}}_{n,m,p}$.
\end{lemma}
\begin{proof}
It was previously established that $\mathcal{G}_{n,m,p}$ is a group. It is a Lie group because it is a Cartesian product of smooth manifolds (i.e., general linear groups and vector spaces of various dimensions) and its multiplication and inversion maps are smooth. Moreover, since the dimension of a product of smooth manifolds is equal to the sum of the factors' dimensions, the dimension of $\mathcal{G}_{n,m,p}$ is $n(n+1)+m(n+1)+m^2 + p(p+1)$ \cite[p. 21]{manifolds}. The group $\mathcal{G}_{n,m,p}$ acts smoothly on $\bar{\mathcal{S}}_{n,m,p}$ since its action involves matrix multiplication and matrix inversion: the former results in every element of the product being a polynomial function of the elements of the factors, while the latter is smooth by Cramer's rule \cite{manifolds}.
\end{proof} 

The next result shows that when the cloud optimizes $\tilde{J}$ and the plant replaces each $y_k$ with output $\tilde{y}_k$, the resulting sequence of inputs ${\tilde{u}}$ can be used to reconstruct a sequence of inputs ${u}$ that optimizes $J$. Its proof amounts to using the change of variables \eqref{eq:tran_state}-\eqref{eq:tran_output}.

\begin{lemma}
\label{th:cost}
Let $\Omega \in \bar{\mathcal{S}}_{n,m,p}$ and $\psi \in \mathcal{G}_{n,m,p}$. Suppose the cloud solves the optimization problem:
\begin{equation*}
\begin{aligned}
& \underset{\tilde{u}}{\text{min}} & \tilde{J}(\tilde{x},\tilde{u})\\
& \text{subject to}
& & \hat{D}\hat{\eta_i} \leq 0, \quad \forall i \in \{0, ..., N \}, 
\end{aligned}
\end{equation*}
for the plant $\tilde{\Sigma}=\psi_*\Sigma$ and the sequence $\tilde{u}^{*}$ is a unique solution of this optimization problem. Then, the unique solution of the optimization problem:
\begin{equation*}
\begin{aligned}
& \underset{u}{\text{min}} & J(x,u)\\
& \text{subject to}
& & D\eta_i \leq 0, \quad \forall i \in \{0, ..., N \}
\end{aligned}
\end{equation*}
for the plant $\Sigma$ is the sequence $u^{*}$ such that \mbox{$u^{*}_i=G^{-1}(\tilde{u}_i^{*}-Fx_i)$} for all $i \in \{0, ..., N \}$.
\end{lemma}

\section{Solving the control-over-the-cloud\\ privacy problem} 

\label{privacy_sol}

\subsection{Enforcing privacy}

The main reason for using isomorphisms is to preclude the cloud from distinguishing between isomorphic systems. We now formalize the notion of indistinguishability.

\begin{definition}
A protocol renders two quadruples $\Omega$ and $\tilde{\Omega}$ indistinguishable by the cloud if the exchanged messages, when using the protocol between the cloud and the plant $\Omega$, and the exchanged messages, when using the protocol between the cloud and the plant $\tilde{\Omega}$, can be made the same.
\end{definition}

The results from Section \ref{iso_sym} allow us to construct a communication protocol between the plant and the cloud that, as will be further shown, solves Problems \ref{scen_1}-\ref{scen_3}. We start by detailing this protocol.
\begin{algorithm}[H]
 \caption{Secure communication}
 \label{com_protocol}
\begin{algorithmic}[1]

 \renewcommand{\algorithmicrequire}{\textbf{Input:}}
 \renewcommand{\algorithmicensure}{\textbf{Output:}}
 
 \REQUIRE Plant: $\psi$, $\Sigma$, $J$, $D$, $\tilde{u}_k$;\\ \hspace{4mm} Cloud: $\tilde{y}_k$, $\tilde{\Sigma}$, $\tilde{J}$, $\tilde{D}$
 \ENSURE Plant: $\tilde{\Sigma}$, $\tilde{J}$, $\tilde{D}$, $\tilde{y}_k$;\\ \hspace{7mm} Cloud: $\tilde{u}_k$
 
\textbf{\textit{Phase 1: Handshaking}}:
 \STATE Plant: Encode $\Sigma$, $J$, $D$ into $\tilde{\Sigma}=\psi_*\Sigma$, $\tilde{J}=\psi_*J$ and $\tilde{D}=\psi_*D$;
 \STATE Plant: Send $\tilde{\Sigma}$, $\tilde{J}$, and $\tilde{D}$ to the cloud;
 \\ \textbf{\textit{Phase 2: Execution}}:
 \STATE Plant: Encode measurement $y_k$ into $\tilde{y}_k=Sy_k$ and send $\tilde{y}_k$ to the cloud;
\STATE Cloud: Use the received $\tilde{y}_k$ to estimate $\tilde{x}_k$ and compute $\tilde{u}_k$ minimizing $\tilde{J}$ subject to the constraints $\tilde{D}$ and the dynamics $\tilde{\Sigma}$;
\STATE Cloud: Send $\tilde{u}_k$ to the plant;
\STATE Plant: Use the isomorphism $\psi$ to decode $\tilde{u}_k$ and produce $u_k$ using \eqref{eq:tran_input};
\STATE Plant: Apply $u_k$ to the actuators.

\end{algorithmic}
\end{algorithm}

From Lemma \ref{th:cost}, we see that Algorithm \ref{com_protocol} provides the plant with the inputs $u_k$ that satisfy the original control objective --- i.e., the plant's trajectory minimizes cost $J$ under affine constraints $D$. 

Let us note how all the required computations in this algorithm are matrix multiplications, which means that both handshaking and execution can be performed in $O(k^3)$ time, where $k=\max\{n,m,p\}$. However, performing matrix multiplications of constant matrices (e.g., $G^{-1}F$) in advance would reduce the complexity of the execution to $O(k^2)$. Both of these complexities were calculated only for the client side (i.e., Plant) of the algorithm.

Let us now show that applying this protocol indeed makes any two systems in the same equivalence class indistinguishable from each other.

\begin{theorem}
\label{th:confusion}
Algorithm \ref{com_protocol} renders isomorphic systems $\Omega =(\Sigma,J,D,\left\{x_k,u_k,y_k\right\}_{k\in \N})$ and $\tilde{\Omega}=(\tilde{\Sigma}, \tilde{J},\tilde{D}, \left\{\tilde{x}_k,\tilde{u}_k,\tilde{y}_k\right\}_{k\in \N})$ indistinguishable by the cloud.
\end{theorem}
\begin{proof}
Since $\Omega$ and $\tilde{\Omega}$ are isomorphic, there exists an isomorphism $\psi$ such that $\psi_* \Sigma = \tilde{\Sigma}$, $\psi_* J = \tilde{J}$, and $\psi_*D = \tilde{D}$. Indistinguishibility of $\Omega$ and $\tilde{\Omega}$ will be shown by running two instances of Algorithm \ref{com_protocol}: one with $\Omega$ and $\psi$ as inputs, the other - with $\tilde{\Omega}$ and the identity isomorphism $\psi_e$. Let us denote the communication algorithm described in Algorithm \ref{com_protocol} applied to $\Omega \in \bar{\mathcal{S}}_{n,m,p}$ with the selected isomorphism $\psi \in \mathcal{G}_{n,m,p}$ by $\text{Alg}(\Omega, \psi)$.
During handshaking:
\begin{itemize}
\item when $\text{Alg}(\Omega, \psi)$ is executed, the plant sends $\psi_* \Sigma$, $\psi_*J$, and $\psi_* D$;
\item when $\text{Alg}(\tilde{\Omega}, \psi_e)$ is executed ($\psi_e$ is the identity of $\mathcal{G}_{n,m,p}$), the plant sends $\tilde{\Sigma}$, $\tilde{J}$, and matrix $\tilde{D}$ unprotected.
\end{itemize} 
Thus, the communicated dynamics and optimization problems are the same.
During execution:
\begin{itemize}
\item when $\text{Alg}(\Omega, \psi)$ is executed, $\psi$ takes trajectories $\left\{ x_k, u_k, y_k\right\}_{k \in \N}$ of $\Sigma$ to trajectories $\left\{ \tilde{x}_k, \tilde{u}_k, \tilde{y}_k \right\}_{k \in \N}$ of $\psi_* \Sigma$;
\item when $\text{Alg}(\tilde{\Omega}, \psi_e)$ is executed, the trajectories are $\left\{ \tilde{x}_k, \tilde{u}_k, \tilde{y}_k \right\}_{k \in \N}$.
\end{itemize}
Therefore, the cloud receives the same measurements from both plants. In response, since both plants communicated the same optimization problem, the cloud sends the same control inputs to both plant $\Omega$ and $\tilde{\Omega}$.
\end{proof}
The result described in Theorem \ref{th:confusion} states that the cloud cannot differentiate between any two plants, costs, constraints or trajectories contained in the same equivalence class of the $\sim_\mathcal{G}$-equivalence relation, thereby protecting the privacy of the system. In the next section, we quantify the amount of privacy provided by Algorithm \ref{com_protocol}.

\subsection{Quantifying privacy}

Privacy is created by preventing the cloud from knowing which quadruple $\Omega$ in its equivalence class $[\Omega]$ it is interacting with. Clearly, the larger the equivalence class, the more privacy is ensured. Since each equivalence class has infinitely many elements, cardinality cannot be used as a measure of privacy. In this section, we show that each equivalence class is a smooth manifold and we 
quantify privacy using the dimension of this manifold.

\subsubsection{Preliminaries: stabilizer subgroups and their dimensions}
The stabilizer subgroup of $\mathcal{G}_{n,m,p}$ for any $\Omega \in \bar{S}_{n,m,p}$, denoted by $\mathcal{K}_{n,m,p}(\Omega)$, is defined by:
\begin{equation}
\mathcal{K}_{n,m,p}(\Omega) = \{ \psi \in \mathcal{G}_{n,m,p} | \psi_{*} \Omega = \Omega \}.
\end{equation}

The subgroup $\mathcal{K}_{n,m,p}(\Omega)$ must be a subset of the symmetry subgroup $\mathcal{K}_{n,m,p}(\Sigma)$ since it must preserve the dynamics.

In \cite{respondek_sym}, Respondek gives a characterization of the symmetries of controllable pairs $(A,B)$.  Since when considering pairs $(A,B)$ the output is not relevant, the isomorphisms of $(A,B)$ degenerate into the form $\phi = (P,F,G)$, where the matrices $P$, $F$ and $G$ are defined to be the same as their counterparts in Definition \ref{def:isomorphism}. We denote the group of these isomorphisms by $\mathcal{G}_{n,m}$. The group action of $\mathcal{G}_{n,m}$ is given by:
\begin{align}
\label{eq:sub_isomorph}
\phi_* (A,B) = (P(A-BG^{-1}F)P^{-1},PBG^{-1}).
\end{align}
Let us define the symmetry subgroup of controllable systems $(A,B)$ as:
\begin{align}
\label{eq:symmetry_subgroup}
\mathcal{K}_{n,m}(A,B) = \{ \phi \in \mathcal{G}_{n,m} | \phi_*(A,B) = (A,B) \}.
\end{align}
The next proposition summarizes the results of \cite{respondek_sym} that are relevant to this paper and complements them with the results from \cite{dimension}:
\begin{proposition}
\label{th:dimension}
Let (A,B) be a controllable pair. Then: 
\begin{align*}
\text{dim }\mathcal{K}_{n,m}(A,B) &= mn - \sum_{i=1}^{r_1} \sum_{j=1}^{\kappa_i-1} r_j + m \\&= m(n+1) - \sum_{i=2}^{m} r_{i-1} r_i,
\end{align*}
where:
\begin{equation}
\nonumber
\begin{aligned}
r_1 &= \text{rank }B, \\
r_i &= \text{rank }S_{i-1}(A,B) - \text{rank }S_{i-2} (A,B),\enspace i = 2, ..., m, \\
S_j&(A,B) = \begin{bmatrix} B & AB & ... & A^{j} B \end{bmatrix}, \enspace j = 1, ..., m -1.
\end{aligned}
\end{equation}
and $\{ \kappa_i \}_{i=1}^m$ are controllability indices of $(A,B)$.
\end{proposition}

This result can be used to estimate the dimension of $\mathcal{K}_{n,m,p}(\Sigma)$. If $\Sigma = (A,B,C)$, then, from Proposition \ref{th:dimension}, we know the dimension of $\mathcal{K}_{n,m}(A,B)$ and that any $\phi \in \mathcal{K}_{n,m}(A,B)$ satisfies $\phi_* (A,B) = (A,B)$. Given $\phi = (P,F,G) \in \mathcal{K}_{n,m}(A,B)$, finding a corresponding $\psi = (P,F,G,S) \in \mathcal{K}_{n,m,p}(\Sigma)$ requires finding $S$ such that $C = SCP^{-1}$. Since we assume $C$ has linearly independent rows, for a given $P$, this equation has at most one solution. A solution exists if and only if $\text{Im }C^T \subset \text{Im }P^{-T}C^{T}$ \cite{LaubMAS}. Let $\mathcal{Q}(A,B,C)$ be the subset of $\mathcal{K}_{n,m}(A,B)$ defined by the elements $(P,F,G)$ for which a unique solution to $C = SCP^{-1}$ exists. It can be seen that there is a one-to-one correspondence between $\mathcal{Q}(A,B,C)$ and $\mathcal{K}_{n,m,p}(\Sigma)$. Since $\mathcal{Q}(A,B,C) \subset \mathcal{K}_{n,m}(A,B)$, this gives an upper bound on the dimension of the symmetry subgroup: 
\begin{align}
\text{dim }\mathcal{K}_{n,m,p}(\Sigma) \leq \text{dim }\mathcal{K}_{n,m,p}(A,B).
\end{align} 

\begin{lemma}
\label{th:dimension_bound}
For any \mbox{$\Omega = (\Sigma,J,D,\left\{x_k,u_k,y_k\right\}_{k\in \N}) \in \bar{\mathcal{S}}_{n,m,p}$},
\begin{align*}
\text{dim }\mathcal{K}_{n,m,p}(\Omega) \leq \text{dim }\mathcal{K}_{n,m,p}(\Sigma) \leq \text{dim }\mathcal{K}_{n,m,p}(A,B),
\end{align*}
where $\text{dim }\mathcal{K}_{n,m,p}(A,B)$ is given by Proposition \ref{th:dimension}.
\end{lemma}

Let us consider a special case, in which the dimension of $\mathcal{K}_{n,m,p}(\Sigma)$ can be computed exactly.
\begin{definition}
A system $\Sigma \in \mathcal{S}_{n,m,p}$ is said to be a prime system if it is $\sim_{\mathcal{G}}$-equivalent to the system of the form:
\begin{equation}
\label{eq:prime_sys}
\begin{aligned}
\Sigma : \quad
\end{aligned}
\begin{cases}
x^{(i,1)}_{k+1} = x^{(i,2)}_k, \\
\vdots \\
x^{(i,\kappa_i)}_{k+1} = u^{(i)}_k, \\
y^{(i)}_k=x^{(i,1)}_{k}, \enspace 1 \leq i \leq m,\\
\end{cases}
\end{equation}
where $x_k = \begin{bmatrix} x^{(1,1)}_k, ..., x^{(1,\kappa_1)}_k, ..., x^{(m,1)}_k, ..., x^{(m,\kappa_m)}_k \end{bmatrix}^T \in \R^n$ and $\{ \kappa_i \}_{i=1}^m$ are controllability indices of $(A,B)$.
\end{definition}
For prime systems we have the following characterization of the dimension of $\mathcal{K}_{n,m,p}(\Sigma)$.
\begin{lemma}
\label{th:dimension_prime}
Let $\Sigma \in \mathcal{S}_{n,m,p}$ be a prime system. Then,
\begin{align}
\text{dim }\mathcal{K}_{n,m,p}(\Sigma) = \sum_{i=1}^m r_{\kappa_i} + m,
\end{align}
where
\begin{equation}
\nonumber
\begin{aligned}
r_1 &= \text{rank }B, \\
r_i &= \text{rank }S_{i-1}(A,B) - \text{rank }S_{i-2} (A,B),\enspace i = 2, ..., m, \\
S_j&(A,B) = \begin{bmatrix} B & AB & ... & A^{j} B \end{bmatrix}, \enspace j = 1, ..., m -1,
\end{aligned}
\end{equation}
and $\{ \kappa_i \}_{i=1}^m$ are controllability indices of $(A,B)$.
\end{lemma}
\begin{proof}
Without loss of generality, let us consider a prime system of the form \eqref{eq:prime_sys}. From Proposition 2 in \cite{respondek_sym}, we can see that if a system is prime, a symmetry $\psi = (P,F,G,S)$ is uniquely defined by a transformation on its outputs (i.e., by transformation $S$).

We want to show that, in order to define a symmetry, transformation $S$ needs to be constructed in such a way that each transformed output $\tilde{y}^{(i)}_k$ is an affine function of outputs $y^{(j)}_k$ with relative degrees greater or equal than that of $y^{(i)}_k$. To simplify notation, we prove this claim for the example with controllability indices $\kappa_1 = \kappa_2 = 2$, $\kappa_3 = 1$, although the employed arguments apply to any prime system:
\begin{alignat}{3}
\nonumber
x^{(1,1)}_{k+1} &= x^{(1,2)}_k  &\qquad  x^{(2,1)}_{k+1} &= x^{(2,2)}_k  &\qquad  x^{(3,1)}_{k+1} &= u^{(3)}_k\\[4 pt]
\label{eq:dynamics}
x^{(1,2)}_{k+1} &= u^{(1)}_k  &  x^{(2,2)}_{k+1} &= u^{(2)}_k &  \\[4 pt]
\nonumber
y^{(1)}_k &= x^{(1,1)}_k  &  y^{(2)}_k &= x^{(2,1)}_k & y_k^{(3)} &= x^{(3,1)}_k. 
\end{alignat}
We will show, by contradiction, that if $S$ produces a transformed output based on outputs of a smaller relative degree, then $S$ cannot be part of a symmetry. In other words, there exist no matrices $P$, $F$, and $G$ such that the quadruple $(P,F,G,S)$ satisfies the equations:
\begin{align}
\label{eq:A}
A &= P(A-BG^{-1}F)P^{-1} \\
\label{eq:B}
B &= PBG^{-1} \\
\label{eq:C}
C &= SCP^{-1}.
\end{align}
Assume that \eqref{eq:A}-\eqref{eq:C} are satisfied and that $S$ contains non-zero elements $S_{ij}$ if $\kappa_i > \kappa_j$ (i.e., the transformed output uses outputs of a smaller relative degree). From \eqref{eq:C}, we have that:
\begin{align}
\label{eq:connect_S}
SCA^{q} B = CPA^{q}B, \quad \forall \ 0 \leq q < \kappa_1.
\end{align}
By using \eqref{eq:A} and \eqref{eq:B}, the following relation can be shown:
\begin{align}
\label{eq:connect_P}
PA = AP+PBG^{-1}F = AP+BF.
\end{align}
Recursively substituting \eqref{eq:connect_P} into \eqref{eq:connect_S} results in:
\begin{align*}
SCA^q B &= C(PA) A^{q-1} B = C(AP + BF) A^{q-1}B\\ &= CBFA^{q-1}B + CAPA^{q-1}B \\
& = CBFA^{q-1}B + CA(PA)A^{q-2}B
 \\&= \hdots \\ &= \sum_{l = 0}^{q-1} CA^l BFA^{q-l-1}B + CA^q PB.
\end{align*}
Equation \eqref{eq:B} implies that $PB = BG$ and, thus, leads to:
\begin{align}
\label{eq:contradiction}
SCA^q B = \sum_{l = 0}^{q-1} CA^l BFA^{q-l-1}B + CA^q BG.
\end{align}
Note that $CA^l B$ is a diagonal matrix such that:
\begin{align}
[CA^l B]_{ii} = 
\begin{cases}
1, \text{ if } \kappa_i = l+1\\
0, \text{ otherwise.}
\end{cases}
\end{align}
In other words, this diagonal matrix marks the indices corresponding to the outputs of equal relative degree. In addition, the expression $FA^{q-l-1}B$ is an $m \times m$ matrix composed out of elements of $F$ (recall that $A$ and $B$ are in the form \eqref{eq:prime_sys}).

The left-hand side of \eqref{eq:contradiction} selects the columns of $S$ corresponding to the outputs of relative degree $\kappa_i = q+1$. For the example in \eqref{eq:dynamics}, taking $q=0$ gives:
\begin{align}
SCB = \begin{bmatrix} 0 & 0 & S_{13} \\ 0 & 0 & S_{23} \\ 0 & 0 & S_{33} \end{bmatrix}. 
\end{align}

The right-hand side of \eqref{eq:contradiction} fills the rows corresponding to the outputs of relative degree smaller or equal than \mbox{$\kappa_i = q+1$} with values from $G$. In case of example in \eqref{eq:dynamics}, the right-hand side, given $q=0$, is:
\begin{align}
CBG = \begin{bmatrix} 0 & 0 & 0 \\ 0 & 0 & 0 \\ \times & \times & \times \end{bmatrix}.
\end{align}

Thus, the equality in \eqref{eq:contradiction}, which was derived using the definition of symmetry, forces $S_{ij}$ to zero if $\kappa_i > \kappa_j$. In the example in \eqref{eq:dynamics}, this leads to $S_{13} = S_{23} = 0$. This contradicts the assumption that $S$ produces a transformed output based on outputs of a smaller relative degree.

This idea can be generalized to any prime system and, therefore, each transformed output $\tilde{y}^{(i)}_k$ can only be an affine function of outputs $y^{(j)}_k$ with relative degrees greater or equal than that of $y^{(i)}_k$.

The number of outputs $y^{(j)}_k$ with a relative degree greater or equal to that of $y^{(i)}_k$ (i.e., greater or equal than $k_i$) is equal to $r_{k_i}$ \cite{dimension}. Therefore, each modified output $y^{(i)}_k$ is an affine function with $r_{k_i}$ arguments and a non-zero constant term, thus leading to the equality:
\begin{align}
\text{dim }\mathcal{K}_{n,m,p} = \sum_{i=1}^{m} (r_{k_i}+1) = \sum_{i=1}^{m} (r_{k_i}) + m.
\end{align}

\end{proof}

\subsubsection{Main results}
\label{quant_privacy}

Consider the scenario from Problem \ref{scen_1}, in which the cloud does not know anything about the system. In this scenario, the plant encodes $\Omega$ using an isomorphism $\psi=(P,F,G,S)$ that can be regarded as a private key used to encode and decode the information exchanged with the cloud. This isomorphism $\psi$ is chosen from $\mathcal{G}_{n,m,p}$, the group of all isomorphisms. 

\begin{proposition}
\label{th:eq_class}
Let $\Omega \in \bar{\mathcal{S}}_{n,m,p}$. Then, under the scenario described in Problem \ref{scen_1}, the cloud cannot distinguish between $\Omega$ and any other system in the uncertainty set $[\Omega]_{\mathcal{G}}$ (i.e., the equivalence class of $\Omega$ defined by the action of $\mathcal{G}_{n,m,p}$) of dimension:
\begin{align}
\text{dim }\mathcal{G}_{n,m,p} - \text{dim }\mathcal{K}_{n,m,p}(\Omega),
\end{align}
if Algorithm \ref{com_protocol} is used.

This implies that the dimension of $[ \Omega ]_{\mathcal{G}}$ is greater or equal than:
\begin{align}
\label{eq:scen_1_bound}
n(n+1)+m^2 +p(p+1) + \sum_{i=2}^m r_{i-1} r_i,
\end{align}
where $r_i$ is given in Lemma \ref{th:dimension}. 

For $\Omega \in \bar{\mathcal{S}}_{n,m,p}$ such that its corresponding $\Sigma \in \mathcal{S}_{n,m,p}$ is prime, this implies that the dimension of $[ \Omega ]_{\mathcal{G}}$ is greater or equal to: 
\begin{align}
\label{eq:scen_1_bound_prime}
n(n+1) + mn +m^2 + p(p+1) - \sum_{i=1}^m r_{\kappa_i},
\end{align}
where $r_{\kappa_i}$ is given in Lemma \ref{th:dimension_prime}.
\end{proposition}
\begin{proof}
From Theorem \ref{th:confusion}, we know that Algorithm \ref{com_protocol} renders isomorphic systems indistinguishable by the cloud. Therefore, the uncertainty set is the set of systems isomorphic to $[\Omega]_{\mathcal{G}}$ - namely, the equivalence class of $\Omega$ defined by the action of $\mathcal{G}_{n,m,p}$.

Let us define a map:
\begin{align*}
\theta_{\Omega}: \mathcal{G}_{n,m,p} &\rightarrow \bar{\mathcal{S}}_{n,m,p}\\
							 \psi &\mapsto \psi_* \Omega.
\end{align*}
Here, $\theta_{\Omega}$ is smooth because, as shown in Lemma \ref{th:lie_group}, $\mathcal{G}_{n,m,p}$ acts smoothly on $\bar{\mathcal{S}}_{n,m,p}$. The stabilizer set can be defined by:
\begin{align*}
\mathcal{K}_{n,m,p}(\Omega) = (\theta_{\Omega})^{-1}(\Omega) = \{ \psi | \psi_* \Omega = \Omega \}.
\end{align*}
Since $\theta_{\Omega}$ and its inverse are smooth and, therefore, continuous, the subgroup $\mathcal{K}_{n,m,p}(\Omega)$ is closed. 

By Theorem 21.17 in \cite{manifolds}, the quotient space $\mathcal{G}_{n,m,p} / \mathcal{K}_{n,m,p}(\Omega)$ is a smooth manifold of dimension $\text{dim }\mathcal{G}_{n,m,p} - \text{dim }\mathcal{K}_{n,m,p}(\Omega)$ such that the quotient map $\pi: \mathcal{G}_{n,m,p} \rightarrow \mathcal{G}_{n,m,p}/\mathcal{K}_{n,m,p}(\Omega)$ is a smooth submersion.

Now, let us define a map:
\begin{align*}
\Theta_{\Omega}: \mathcal{G}_{n,m,p}/\mathcal{K}_{n,m,p}(\Omega) &\rightarrow \bar{\mathcal{S}}_{n,m,p}\\ 
							 \psi \mathcal{K}_{n,m,p}(\Omega) &\mapsto \psi_* \Omega,
\end{align*}
where $\psi \mathcal{K}_{n,m,p}(\Omega)$ is a left coset of $\mathcal{K}_{n,m,p}(\Omega)$. It can be shown that $\Theta_{\Omega}$ is well-defined. 


By Theorem 4.29 in \cite{manifolds}, $\Theta_{\Omega}$ is smooth because $\theta_{\Omega}=\Theta_{\Omega} \circ \pi$ is smooth and $\pi$ is a smooth submersion.

It can be shown that the map $\Theta_{\Omega}$ is equivariant (see \cite[p. 164]{manifolds}) and, therefore, by the equivariant rank theorem \cite[p. 165]{manifolds}, we have that $\Theta_{\Omega}$ has a constant rank.

Let us show that $\Theta_{\Omega}$ is injective. If $\Theta_{\Omega}(\psi_1 \mathcal{K}_{n,m,p}(\Omega)) = \Theta_{\Omega}(\psi_2 \mathcal{K}_{n,m,p}(\Omega))$, then $(\psi_1)_* \Omega = (\psi_2)_* \Omega$. This implies that $(\psi_1)^{-1} \psi_2 \in \mathcal{K}_{n,m,p}(\Omega)$ and, therefore, $\psi_1 \mathcal{K}_{n,m,p}(\Omega) = \psi_2 \mathcal{K}_{n,m,p}(\Omega)$. Therefore, $\Theta_{\Omega}$ is a smooth immersion.

By Proposition 5.18 in \cite{manifolds}, the image of $\Theta_{\Omega}$ (i.e., the equivalence class $[\Omega]_{\mathcal{G}}$) is an immersed submanifold such that $\Theta_{\Omega}: \mathcal{G}_{n,m,p}/\mathcal{K}_{n,m,p}(\Omega) \rightarrow [\Omega]_{\mathcal{G}}$ is a diffeomorphism and, therefore, the dimension of $[\Omega]_{\mathcal{G}}$ is equal to the dimension of $\mathcal{G}_{n,m,p}/\mathcal{K}_{n,m,p}(\Omega)$.



A more concrete quantification of privacy can be given for various special cases. Using the results of Proposition \ref{th:dimension} and Lemma \ref{th:dimension_bound}, we have that, for any $\Omega \in \bar{\mathcal{S}}_{n,m,p}$, the uncertainty sets under the scenario described in Problem \ref{scen_1} are smooth manifolds of dimension greater or equal to the value in \eqref{eq:scen_1_bound}

The dimension of the uncertainty sets can be shown to be greater or equal to the value in \eqref{eq:scen_1_bound_prime} using Lemma \ref{th:dimension_prime}.
\end{proof}

Proposition \ref{th:eq_class} can be used to quantify privacy of other scenarios presented in Section \ref{problem}.

Consider the scenario in Problem \ref{scen_2}, where the cloud does not know the dynamics but knows which sensors and actuators will be used. An arbitrary isomorphism can no longer be used for encoding since it could lead to inputs and outputs that are inconsistent with existing sensors and actuators. This inconsistency would signal the cloud that the plant is being dishonest about its measurements and provide the cloud with an opportunity to exploit this fact to gather additional knowledge. Therefore, we need to restrict the group of isomorphisms used for encoding. These isomorphisms are given by any composition of $\psi_1 = (P,0,I,I)$ for any $P \in GL(n, \R)$ and $\psi_2 \in \mathcal{K}_{n,m,p}(\Sigma)$. It can be shown that this set of isomorphisms forms a subgroup that we denote by $\mathcal{H}_{n,m,p}(\Sigma) \subset \mathcal{G}_{n,m,p}$. 

\begin{corollary}
\label{th:scen_2}
Let $\Omega \in \bar{\mathcal{S}}_{n,m,p}$. Then, under the scenario described in Problem \ref{scen_2}, the cloud cannot distinguish between $\Omega$ and any other system in the uncertainty set $[\Omega]_{\mathcal{H}}$ (i.e., the equivalence class of $\Omega$ defined by the action of $\mathcal{H}_{n,m,p}$) of dimension:
\begin{align}
\text{dim }\mathcal{H}_{n,m,p}(\Sigma) - \text{dim } \mathcal{K}_{n,m,p}(\Omega),
\end{align}
if Algorithm \ref{com_protocol} is used. This implies that the dimension of $[\Omega]_{\mathcal{H}}$ is greater or equal to $n(n+1)$.
\end{corollary}
\begin{proof}
From Theorem \ref{th:confusion}, we know that Algorithm \ref{com_protocol} renders isomorphic systems indistinguishable by the cloud. However, the uncertainty set is no longer the equivalence class under the entire group of isomorphisms $\mathcal{G}_{n,m,p}$, but the equivalence class under a smaller group $\mathcal{H}_{n,m,p}(\Sigma)$ denoted by $[\Omega]_{\mathcal{H}}$.

It can be shown that $\mathcal{H}_{n,m,p}(\Sigma)$ is a Lie subgroup of $\mathcal{G}_{n,m,p}$. This subgroup $\mathcal{H}_{n,m,p}(\Sigma)$ can be thought of as a product manifold of $\mathcal{K}_{n,m,p}(\Sigma)$ and a space of invertible affine maps. Since the dimension of a product manifold is a sum of its factors' dimensions, we have: $$\text{dim }\mathcal{H}_{n,m,p}(\Sigma) = \text{dim }\mathcal{K}_{n,m,p}(\Sigma) + n(n+1).$$ The result follows by applying Proposition \ref{th:eq_class} to $\mathcal{H}_{n,m,p}(\Sigma)$. Using the result from Lemma \ref{th:dimension_bound}, we can see that the dimension of the uncertainty set for any $\Omega \in \bar{\mathcal{S}}_{n,m,p}$ is greater or equal to $n(n+1)$.
\end{proof}

Finally, in the scenario described in Problem \ref{scen_3}, where the cloud possesses the complete knowledge of dynamics, only the isomorphisms from the symmetry subgroup $\psi \in \mathcal{K}_{n,m,p}(\Sigma)$ can be used. To provide privacy guarantees for this scenario, let us assume that we have $n+1$ linearly independent constraints on the state $x_k$ expressed by the constraint matrix $D$. This is a reasonable assumption because systems often have an operational envelope bounding the states. Therefore, any $\psi \in \mathcal{K}_{n,m,p}(\Omega)$ must satisfy:
\begin{equation}
\nonumber
\begin{aligned}
DL^{-1} = D & \Longleftrightarrow DL = D \\ & \Longleftrightarrow \begin{bmatrix}D_{11} & 0 \\ D_{21} & D_{22} \end{bmatrix} \begin{bmatrix} P & 0 \\ F & G \end{bmatrix} = \begin{bmatrix}D_{11} & 0 \\ D_{21} & D_{22} \end{bmatrix}\\  &\Longrightarrow D_{11} P = D_{11}.
\end{aligned}
\end{equation}
Given that $D_{11} \in \R^{h_1 \times (n+1)}$ is injective, the last equality is satisfied if and only if $P=I$. Since $P$ uniquely defines $F$, $G$ and $S$, we also have that the only isomorphism that keeps $(A,B,C,D_{11})$ invariant is $\psi = \psi_e=(I,0,I,I)$ . Therefore, the only element of $\mathcal{K}_{n,m,p}(\Omega)$ is $\phi_{e} = (I,0,I,I)$ and $\text{dim }\mathcal{K}_{n,m,p}(\Omega) = 0$.

\begin{corollary}
\label{th:scen_3}
Let $\Omega \in \bar{\mathcal{S}}_{n,m,p}$. Then, under the scenario described in Problem \ref{scen_3}, the cloud cannot distinguish between $\Omega$ and any other system in the uncertainty set $[\Omega]_{\mathcal{K}}$ (i.e., the equivalence class of $\Omega$ defined by the action of $\mathcal{K}_{n,m,p}$) of dimension:
\begin{align}
\text{dim }\mathcal{K}_{n,m,p}(\Sigma) - \text{dim } \mathcal{K}_{n,m,p}(\Omega),
\end{align}
if Algorithm \ref{com_protocol} is used.

When the constraint matrix $D$ contains $n+1$ linearly independent constraints on the state, the dimension of the uncertainty set is equal to $\text{dim }\mathcal{K}_{n,m,p}(\Sigma)$, which is greater or equal to: $$m(n+1) - \sum_{i=2}^{m} r_{i-1}r_i$$.

Moreover, for any $\Omega \in \bar{\mathcal{S}}_{n,m,p}$ such that its corresponding $\Sigma \in \mathcal{S}_{n,m,p}$ is prime, the dimension of $[\Omega]_{\mathcal{K}}$ is equal to $$\sum_{i=1}^{m}r_{k_i} + m$$.
\end{corollary}
\begin{proof}
The proof of this statement is similar to that of Corollary \ref{th:scen_2}. The dimensions of equivalence classes for prime and general systems were evaluated using results of Proposition \ref{th:dimension} and Lemma \ref{th:dimension_prime}.
\end{proof}

\section{Side knowledge}

The privacy guarantees derived in Section \ref{privacy_sol} are compromised when the adversary has partial information about the encoding isomorphism. In our problem formulation, we assume that the cloud may have learned those through some external channels or through some prior knowledge about the system.

Recall that by Lemma \ref{th:lie_group}, $\mathcal{G}_{n,m,p}$ is a Lie group of dimension $n(n+1)+m(n+1)+m^2 + p(p+1)$. In this section, we assume that the constraint matrix $D$ has $n+1$ linearly independent constraints on the state and, therefore, as shown in the previous section, $\mathcal{K}_{n,m,p}(\Omega) = \{ \psi_e \}$, where $\psi_e$ is the identity element of $\mathcal{G}_{n,m,p}$. 

Suppose the cloud has partial knowledge about the encoding isomorphism. We shall represent the partial knowledge available to the cloud as a projection from $\mathcal{G}_{n,m,p}$ onto a $k$-dimensional vector space. Let us define $\rho: \mathcal{G}_{n,m,p} \rightarrow \R^k$ to be a surjective map of constant rank $k$, providing side knowledge about the encoding isomorphism. Then, we can say that the cloud knows some vector $l \in \R^{k}$, where:
\begin{align}
\label{eq:partial}
l=\rho (P,F,G,S).
\end{align}
Note that this map is not known to us, and the results that follow do not require the knowledge of this map.

Side knowledge does not change the result of Theorem \ref{th:confusion}, however the privacy guaranteed by the scheme changes. It is obvious that the size of the uncertainty set defined by isomorphisms that satisfy \eqref{eq:partial} is no greater and, in general, smaller thanif no side knowledge is available. Moreover, the uncertainty set is no longer neither an orbit nor an equivalence class because the preimage of $\rho$ does not necessarily have a group structure. 


Let us show that the object defined by \eqref{eq:partial} on $\mathcal{G}_{n,m,p}$ is still a manifold.

\begin{lemma}
\label{th:iso_conf}
Let $\mathcal{G}_{n,m,p}$ be the group of all isomorphisms, $\rho: \mathcal{G}_{n,m,p} \rightarrow \R^k$ be a surjective map of constant rank $k$ and assume the cloud knows that $l=\rho(P,F,G,S)$. Then, $\rho^{-1}(l)$, representing the possible encoding isomorphisms used by the client, is a properly embedded submanifold of $\mathcal{G}_{n,m,p}$. Its dimension is $\text{dim } \mathcal{G}_{n,m,p} - k$.
\end{lemma}
\begin{proof}
By the global rank theorem \cite[p. 83]{manifolds}, since $\rho$ is a surjective map of constant rank $k$, it is a smooth submersion. From the submersion level set theorem \cite[p. 105]{manifolds}, since both $\mathcal{G}_{n,m,p}$ and $\R^{k}$ are smooth manifolds and $\rho$ is a smooth submersion, we have that $\rho^{-1}(l)$ is a properly embedded submanifold of dimension $\text{dim }\mathcal{G}_{n,m,p} - \text{dim } \R^k = n(n+1)+m(n+1)+m^2 + p(p+1) - k$.
\end{proof}

Let us now consider the map $\Theta_{\Omega}$ defined earlier in Proposition \ref{th:eq_class}. Since $\mathcal{K}_{n,m,p}(\Omega) = \psi_e$, we have that $\mathcal{G}_{n,m,p} / \mathcal{K}_{n,m,p}(\Omega)$ is equivalent to $\mathcal{G}_{n,m,p}$. Therefore, the map $\Theta_{\Omega}$ is equivalent to the orbit map $\theta_{\Omega}$. It was shown in Proposition \ref{th:eq_class} that $\Theta_{\Omega}$ is injective. The image of $\Theta_{\Omega}(\rho^{-1}(l))$ constitutes the uncertainty set, between the elements of which the cloud is not be able to distinguish. Therefore, the main result of this section requires finding the dimension of $\Theta_{\Omega}(\rho^{-1}(l))$.

\begin{proposition}
\label{th:side_knowledge}
Assume $\Omega \in \bar{\mathcal{S}}_{n,m,p}$ is such that the constraint matrix $D$ has $n+1$ linearly independent constraints on the state. Suppose that Algorithm \ref{com_protocol} is used and the cloud has the following side knowledge about the selected isomorphism $\psi$:
$$\rho(P,F,G,S) = l \in \R^k,$$
where $\rho: \mathcal{G}_{n,m,p} \rightarrow \R^k$ is a surjective map of constant rank $k$.
Then, under the scenario described in Problem \ref{scen_1}, the cloud cannot distinguish between $\Omega$ and any other system in the uncertainty set $\mathcal{U}=\Theta_{\Omega}(\rho^{-1}(l) )$ of dimension:
\begin{align}
\text{dim } \mathcal{G}_{n,m,p} - k = n(n+1)+m(n+1)+m^2 + p(p+1) - k.
\end{align}
\end{proposition}
\begin{proof}
By Theorem \ref{th:confusion}, Algorithm \ref{com_protocol} renders isomorphic systems indistinguishable by the cloud. However, the cloud knows that we use an isomorphism $\psi \in \rho^{-1}(l)$ and, therefore, the uncertainty set is no longer the equivalence class under the entire group of isomorphisms $\mathcal{G}_{n,m,p}$, but the subset of this equivalence class $\mathcal{U} = \Theta_{\Omega}(\rho^{-1}(l) )$.

By the property of the orbit map \cite[p. 166]{manifolds}, for each $\Omega$, the orbit map $\Theta_{\Omega}$ is smooth and has constant rank. Since $\Theta_{\Omega}$ is also injective, we have, by the Global Rank Theorem, that it is a smooth immersion \cite[p. 83]{manifolds}. As it was shown in Lemma \ref{th:iso_conf}, the set $\rho^{-1}(l)$ is an embedded submanifold of $\mathcal{G}_{n,m,p}$ and, therefore, the inclusion map $i: \rho^{-1}(l) \rightarrow \mathcal{G}_{n,m,p}$ is a smooth embedding. 

The map $\Theta_{\Omega} \circ i$ is a smooth immersion because it is a composition of smooth immersions. Since images of smooth immersions are smooth immersed submanifolds (by Proposition 5.18 from \cite{manifolds}), the uncertainty set $\mathcal{U} = \Theta_{\Omega}(\rho^{-1}(l))$ is a smooth immersed submanifold of $\bar{\mathcal{S}}_{n,m,p}$ diffeomorphic to $\rho^{-1}(l)$ and, hence, has the same dimension (refer to Lemma \ref{th:iso_conf}).

Using Lemma \ref{th:lie_group}, the dimension of the uncertainty set is evaluated to be:
\begin{align*}
n(n+1)+m(n+1)+m^2 + p(p+1) - k.
\end{align*}
\end{proof}
Remark: although Proposition \ref{th:side_knowledge} was proved under the assumption that $D$ has $n+1$ linearly independent constraints on the state, this assumption can be dropped if we assume the intersection of $\rho^{-1}(l)$ and the left cosets of $\mathcal{K}_{n,m,p}(\Omega)$ in $\mathcal{G}$ is well-behaved.

This result shows that the proposed scheme degrades gracefully with side knowledge --- i.e., side knowledge allows the cloud to reduce the dimension of the uncertainty set only by the amount of side knowledge and not more. Moreover, this result can be generalized for other scenarios considered in Section \ref{quant_privacy} using similar proofs.

\begin{corollary}
Assume $\Omega \in \bar{\mathcal{S}}_{n,m,p}$ is such that the constraint matrix $D$ has $n+1$ linearly independent constraints on the state. Suppose that Algorithm \ref{com_protocol} is used and the cloud has the following side knowledge $l \in \R^k$ about the selected isomorphism $\psi$:
$$l=\rho(P,F,G,S),$$
where $\rho: \mathcal{G}_{n,m,p} \rightarrow \R^k$ is a surjective map of constant rank $k$.
Then, under the scenario described in Problem \ref{scen_2}, the cloud cannot distinguish between $\Omega$ and any other system in the uncertainty set $\mathcal{U}= \Theta_{\Omega}(\rho^{-1}(l) )$ of dimension:
\begin{align}
\text{dim } \mathcal{H}_{n,m,p}(\Sigma) - k.
\end{align}
Under the scenario described in Problem \ref{scen_3}, the dimension of the uncertainty set is:
\begin{align}
\text{dim } \mathcal{K}_{n,m,p}(\Sigma) - k.
\end{align}
\end{corollary}



\section{Conclusion}

In this paper, we proposed a transformation-based method to preserve privacy in control over the cloud. In addition to its low computational overhead, we have formally shown that this method precludes the adversary from inferring the private data by eavesdopping on the messages exchanged between the plant and the cloud. We quantified the guaranteed privacy via the dimension of the set that describes the uncertainty experienced by the adversary. The problem of computing the dimension of the stabilizer set $\mathcal{K}_{n,m,p}(\Omega)$ remains open, and its solution requires a detailed analysis of system-theoretic properties. The authors are currently investigating other measures of privacy that may lead to a deeper insight into the proposed method.


\bibliographystyle{abbrv}
\bibliography{paper}
\end{document}